\documentclass{amsart}
\usepackage{amsmath,amsfonts,amsthm}
\usepackage{amssymb,latexsym}
\usepackage{graphics}
\usepackage[colorlinks]{hyperref}

\usepackage{amssymb}

\theoremstyle{plain}
\theoremstyle{definition}
\newtheorem{theorem}{Theorem}%[section]
\newtheorem{lemma}[theorem]{Lemma}

\newtheorem{corollary}[theorem]{Corollary}

\newtheorem{example}[theorem]{Example}

\newtheorem{remark}[theorem]{Remark}
\theoremstyle{remark}

\numberwithin{equation}{section}

\title{Closed graph property in Alexandroff spaces}
\author[F. Ayatollah Zadeh Shirazi, S. Moradi Chaleshtori]{Fatemah Ayatollah Zadeh Shirazi, Sajjad Moradi Chaleshtori}
\begin{document}
%%%%%%%%%%%%%%%%%%%%%%%%%%%%%%%%%%%% abstract
\begin{abstract}
In the following text we show if $X$ is an Alexandroff space, then $f:X\to Y$
has closed graph if and only if it has constant closed value on each connected component of $X$.
Moreover, if $X$ an Alexandroff space and $f:X\to Y$
has closed graph, then $f:X\to Y$ is continuous. As a matter of fact, the
number of maps which have closed graph from Alexandroff space $X$ to a topological  space $Y$
depends just on the the number of connected components of $X$ and the number of closed points of $Y$.
\end{abstract}
\maketitle
%%%%%%%%%%%%%%%%%%%%%%%%%%%%%%%%%%%% MSC
\noindent {\small {\bf 2020 Mathematics Subject Classification:}  54C10, 54C35   \\
{\bf Keywords:}} Alexandroff space, closed graph, connected component
%%%%%%%%%%%%%%%%%%%%%%%%%%%%%%%%%%%%
\section{Introduction}
\noindent For arbitrary map $f:X\to Y$, $G_f:=\{(x,f(x)):x\in X\}$ denotes the graph of 
$f:X\to Y$  \cite{bana}. For topological spaces $X,Y$ we say $f:X\to Y$ has closed graph 
if $G_f$ is a closed subset of $X\times Y$. Studying maps with closed graph is the
main idea of many texts in different categories, in topological groups \cite{gr, husain},
topological linear spaces \cite{tl}, more general spaces like Hausdorff spaces \cite{T2}. 
In this text we study self--maps on Alexandroff spaces with closed graphs.
Alexandroff spaces are topological spaces with  an extra property, namely, 
the intersection of every nonempty family of open sets is open. These spaces called 
Alexandroff spaces in the honour of  
\linebreak
P. Alexandroff's famous paper on~1937~\cite{alex}. 
%%%%%%%%%%%%%%%%%%%%%%%%%%%%%%%%%%%%%%
\subsection*{More details on Alexandroff spaces}
As a matter of fact a topological space $X$ is an Alexandroff space if and only if each $ x\in X $ has the smallest open neighbourhood. 
In Alexandroff space $X$ let's denote the smallest open neighbourhood of $a\in X$ by  $V_a$. 
\\
In Alexandroff space $X$, $f : X\to \mathcal{P}(X)$ with
$f(x)=V_x$ for $x\in X$ (where $\mathcal{P}(X)$ denotes the collection of subsets of $X$)
satisfies the following properties:
\begin{enumerate}
\item $\forall x \in X \:(x \in f(x))$
\item $\forall x \in X \: \forall y \in f(x) \:(f(y) \subseteq f(x))$.
\end{enumerate}
Conversely, if there exists a function $f : X\to \mathcal{P}(X)$ which satisfies the above two
conditions, then $f(X)$ is a base for an Alexandroff topology on X.
%%%%%%%%%%%%%%%%%%%%%%%%%%%%%%%%%%%%%%
%%%%%%%%%%%%%%%%%%%%%%%%%%%%%%%%%%%%%%
\\
For $f:X\to Y$ and $A\subseteq X$, $f\restriction_A:A\to Y$ denotes the restriction of
$f$ to $A$. 
%%%%%%%%%%%%%%%%%%%%%%%%%%%%%%%%%%%%%%
%%%%%%%%%%%%%%%%%%%%%%%%%%%%%%%%%%%%
\section{Connected components of Alexandroff spaces}
\noindent In Alexandroff space $X$ consider the equivalence relation 
$\Re_X:=\{(x,y)\in X\times X:$ there exist $x=x_1,x_2,\ldots,y=x_n\in X$ such that
$V_{x_i} \cap V_{x_{i+1}} \neq\emptyset $ for all 
$i \in \{1,2,..., n-1\}\:\}$ (see \cite[Lemma 2.1]{attar} too). Also let's recall that
for each  $x \in X$,  $[x]_{\Re_X}:=\{ y\in X: (x,y)\in\Re_X  \}$ is
the equivalence class of $x$. 
In this section we prove that connected components of Alexandroff space $X$ are just
equivalence classes of $\Re_X$.%%%%%%%%%%%%%%%%%%%%%%%%%%%%%%%%%%%%
\begin{lemma}\label{vx}
In Alexandroff space $X$ for each $x\in X$, $[x]_{\Re_X}$ is open and a subset of connected 
component of $X$ containing $x$.
\end{lemma}
%%%%%%%%%%%%%%%%%%%%%%%%%%%%%%%%%%%%
\begin{proof}
Choose arbitrary $x\in X$. Using the definition of $\Re_X$, it's evident that 
$ V_x \subseteq [x]_{\Re_X}$, which leads us to
$[x]_{\Re_X}= \bigcup\{V_z:z \in  [x]_{\Re_X}\}$, in particular 
$[x]_{\Re_X}$ is open.
\\
Note that $V_x$ is connected, otherwise there exist disjoint nonempty open subsets $U,V$
of $V_x$ (hence disjoint nonempty open subsets of $X$) such that $x\in V_x=U\cup V$,
we may suppose $x\in U$ which is a contradiction (since $V_x$ is the smallest
open neighbourhood of $x$ and $U$ is a proper open subset of $V_x$). 
\\
Suppose $C_x$ is the connected component of $X$ which
contains $x$. 
Now using the definition of $\Re_X$ if $y\in[x]_{\Re_X}$, then
there exist $x=x_1,\ldots,x_n=y\in X$ such that $V_{x_i}\cap V_{x_{i+1}}\neq\varnothing$
for each $i\in\{1,\ldots,n-1\}$. Using the fact that union of two connected set with 
nonempty intersection is connected too, we have the following chain of connected subsets:
\[V_{x_1}\subseteq V_{x_1}\cap V_{x_2}\subseteq\cdots\subseteq V_{x_1}\cap V_{x_2}\cap\cdots\cap V_{x_n}\:.\]
Hence $C=V_{x_1}\cap V_{x_2}\cap\cdots\cap V_{x_n}$ is a connected subset of $X$ containing
$x=x_1,y=x_n$, therefore $y\in C\subseteq C_x$ which leads to $[x]_{\Re_X}\subseteq C_x$.
\end{proof}
%%%%%%%%%%%%%%%%%%%%%%%%%%%%%%%%%%%%
\noindent The following theorem shows $[x]_{\Re_X}$ is just
the connected component of $X$ containing $x$, for each $x$ in Alexandroff space $X$. Also
equivalence classes of $X$ with respect to $\Re_X$ are exactly its connected components.
%%%%%%%%%%%%%%%%%%%%%%%%%%%%%%%%%%%%
\begin{theorem}\label{connected component}
In Alexandroff space $X$ for each $x\in X$, the connected component of $X$ which contains $x$ is $[x]_{\Re_X}$.
\end{theorem}
%%%%%%%%%%%%%%%%%%%%%%%%%%%%%%%%%%%%
\begin{proof}
Suppose that $C_x$ is the connected component of $X$ which contains $x$. 
 If $[x]_{\Re_X}=X$, then $[x]_{\Re_X}=X=C_x$, otherwise by Lemma~\ref{vx}, $[x]_{\Re_X}$ 
 and 
 \linebreak $\bigcup\{[y]_{\Re_X}:y\in X\setminus[x]_{\Re_X}\}=X\setminus[x]_{\Re_X}$ are two disjoint open sets, hence
 $[x]_{\Re_X}$, 
 \linebreak $X\setminus [x]_{\Re_X}$
 is a separation 
 of $X$, thus either $C_x\subseteq [x]_{\Re_X}$ or $C_x\subseteq X\setminus[x]_{\Re_X}$, considering
$x\in C_x\cap [x]_{\Re_X}$ leads to $C_x\cap [x]_{\Re_X}\neq\varnothing$ 
and $C_x\subseteq [x]_{\Re_X}$. By Lemma~\ref{vx}, $[x]_{\Re_X}\subseteq C_x$ which completes the proof.
\end{proof}
%%%%%%%%%%%%%%%%%%%%%%%%%%%%%%%%%%%%
\section{Maps between Alexandroff spaces and closed graph property}
\noindent In this section
we show if $X$ is an Alexandroff space and $f:X\to Y$ has closed graph, then $f:X\to Y$ is continuous. Moreover, in Alexandroff space $X$ the number of maps from $X$ to $Y$ which have closed graph just depends on the number of
connected components of $X$ and the number of closed points of $Y$. 
As a matter of fact, for Alexandroff space $X$ a map 
$f:X\to Y$ has closed graph if and only if $f$ maps each connected
component of $X$ to a closed point of $Y$.
%%%%%%%%%%%%%%%%%%%%%%%%%%%%%%%%%%%%
\begin{remark}\label{salam10}
In topological spaces $X,Y$ suppose $x\in X$ and $f:X\to Y$ has closed graph then for
each $x\in X$ we have 
\begin{itemize}
\item[(1)] $\{x\}\times \overline{\{f(x)\}}\subseteq \overline{\{(x,f(x))\}}\subseteq \overline{G_f}=G_f$, thus $ \overline{\{f(x)\}}=\{f(x)\}$ and $f(x)$ is a closed point of $Y$,
\item[(2)] $\overline{\{x\}}\times\{f(x)\}\subseteq \overline{\{(x,f(x))\}}\subseteq \overline{G_f}=G_f$, thus $f\restriction_{\overline{\{x\}}}:\overline{\{x\}}\to Y$ is the constant map with value $f(x)$ \cite[Lemma 2.2]{attar}, 
\item[(3)] if $z\in\bigcap\{V:V$ is an open neighbourhood of $x\}=:W$, then $x\in \overline{\{z\}}$, thus $f(x)=f(z)$ by item(2) hence $f\restriction_{W}:W\to Y$ is the constant map with value $f(x)$ too,
\item[(4)] in particular if $X$ is an Alexandroff space, then in item (3), $W=V_x$ and 
$f\restriction_{\overline{\{x\}}\cup V_x}:\overline{\{x\}}\cup V_x\to Y$ is the constant map with value $f(x)$.
\end{itemize}
\end{remark}
%%%%%%%%%%%%%%%%%%%%%%%%%%%%%%%%%%%%  
\begin{theorem}\label{salam15}
Let $X$ be an Alexandroff space and $f: X \rightarrow Y$ has closed graph, then $f: X \rightarrow Y$ is continuous.
\end{theorem}
%%%%%%%%%%%%%%%%%%%%%%%%%%%%%%%%%%%%  
\begin{proof}
By item (4) in Remark \ref{salam10} for each $x\in X$, $V_x\subseteq f^{-1}(f(x))$ thus
for each $D\subseteq Y$,
$f^{-1}(D)=\bigcup\{V_x:f(x)\in D\}$ is an open subset of $X$.
\end{proof}
%%%%%%%%%%%%%%%%%%%%%%%%%%%%%%%%%%%%
\begin{lemma}\label{consclosgph}
Let $X$ be an Alexandroff space, $x\in X$ and $f:X\rightarrow Y$ has closed graph, then $f\restriction_{[x]_\Re}:[x]_\Re\to Y$ is a constant map.
\end{lemma}
%%%%%%%%%%%%%%%%%%%%%%%%%%%%%%%%%%%%
\begin{proof}
By Remark~\ref{salam10}, for each $z\in X$, 
$f \restriction_{V_z}:V_z\to Y$ is constant. If $y\in[x]_{\Re_X}$, then there exist 
$x=x_1,\ldots, x_n=y$ in $X$ such that $V_{x_i}\cap V_{x_{i+1}} \neq \emptyset$
for each $i=1, \ldots, n-1$. So, $f(x)=f(x_1)= f(x_2)= \cdots= f(x_n)=f(y)$.  
Therefore $f\restriction_{[x]_\Re}:[x]_\Re\to Y$ is a constant map.
\end{proof}
%%%%%%%%%%%%%%%%%%%%%%%%%%%%%%%%%%%%  
\begin{lemma}\label{salam30}
Let $X$ be   an Alexandroff space. Then $f: X \rightarrow Y$ has closed graph if and only if for every connected component $C$ of $X$, $f(C)$ is a singleton closed subset  of $Y$.
\end{lemma}
%%%%%%%%%%%%%%%%%%%%%%%%%%%%%%%%%%%%  
\begin{proof}
Suppose $f: X \rightarrow Y$ has closed graph and $C$ is a connected component of $X$. By Theorem \ref{connected component}, for each $x\in C$, $C=[x]_{\Re_X}$. So Lemma \ref{consclosgph} shows that $f$ is constant on $C$, by Remark \ref{salam10}, $f(C)=\{f(x)\}$
is closed too. 
\\
Now, suppose for every connected component $C$ of $X$, $f(C)$ is a singleton closed subset of $Y$. If $(z,w)\in\overline{G_f}$, then there exists a net $\{x_\alpha\}_{\alpha\in\Gamma}$
in $X$ such that $\{(x_\alpha,f(x_\alpha)\}_{\alpha\in\Gamma}$ converges to $(z,w)$.
Hence  $\{x_\alpha\}_{\alpha\in\Gamma}$ converges to $z$ and there exists $\beta\in\Gamma$
such that $x_\alpha\in V_z$ for each $\alpha\geq\beta$. Since $V_z$ is a subset of
connected component containing $z$, $f$ is constant on $V_z$, so $f(x_\alpha)=f(z)$
for each $\alpha\geq\beta$. Since $\{f(x_\alpha)\}_{\alpha\in\Gamma}$ converges to
$w$, $\{f(z)\}_{\alpha\in\Gamma,\alpha\geq\beta}=\{f(x_\alpha)\}_{\alpha\in\Gamma,\alpha\geq\beta}$ converges to $w$ too, hence $w\in\overline{\{f(z)\}}=\{f(z)\}$. Therefore
$(z,w)=(z,f(z))\in G_f$ and $G_f$ is a closed subset of $X\times Y$.
\end{proof}
%%%%%%%%%%%%%%%%%%%%%%%%%%%%%%%%%%%%  
%%%%%%%%%%%%%%%%%%%%%%%%%%%%%%%%%%%%  
\begin{theorem}[main]\label{main}
Consider Alexandroff space $X$ and arbitrary map $f:X\to Y$. The following statements are equivalent:
\\
$\bullet$ $f:X\to Y$ has closed graph,
\\
$\bullet$ for every connected component $C$ of $X$, $f(C)$ is a singleton closed subset  of $Y$,
\\
$\bullet$ for every $x\in X$, $f([x]_{\Re_X})$ is a singleton closed subset  of $Y$.
\\
Moreover, in the above case $f:X\to Y$ is continuous.
\end{theorem}
%%%%%%%%%%%%%%%%%%%%%%%%%%%%%%%%%%%%  
\begin{proof}
Use Theorems \ref{connected component}, \ref{salam15} and Lemma \ref{salam30}.
\end{proof}
%%%%%%%%%%%%%%%%%%%%%%%%%%%%%%%%%%%%  
\begin{corollary}
Suppose $X$ is an Alexandroff space and $Y$ is an arbitrary topological space, let:
\begin{itemize}
\item[] $\alpha:=card\{C:C$ is a connected component of $X\}(=card(\dfrac{X}{\Re_X}))$,
\item[] $\beta:=card\{y:y$ is a closed point of $Y\}$.
\end{itemize}
Then by Theorem \ref{main} we have (by $\mathcal{C}(X,Y)$ we mean the collection of continuous maps from $X$ to $Y$, also
$Y^X$ denotes the collection of all maps from $X$ to $Y$):
\begin{center}
$card\{ f\in Y^X:f$ has closed graph$\}=\beta^\alpha\leq card(\mathcal{C}(X,Y))$.
\end{center}
\end{corollary}
%%%%%%%%%%%%%%%%%%%%%%%%%%%%%%%%%%%%
\noindent Compare the following (counter)example by Theorem \ref{salam15}.
%%%%%%%%%%%%%%%%%%%%%%%%%%%%%%%%%%%%
\begin{example}
Suppose $X$ is a connected Alexandroff space, then by Lemma \ref{salam30}, $f:X\to X$ has closed graph if and only if
it is a constant map with closed point value. 
Equip $\mathbb{Z}=\{0,\pm1,\pm2,\cdots\}$ with 
topology generated by the base consist of  sets of the form 
$\{ 2m+1 \}$ and $\{ 2m-1, 2m, 2m+1 \}$ in which $m$ varies arbitrarily in $\mathbb Z$, we call this space Khalimsky line and denote it by $\mathcal{K}$. For every $n\in \mathbb{N} $, $\mathcal{K}^n$ with product topology is  called 
$n$ dimensional Khalimsky space.
$\mathcal{K}^n$ is a connected Alexandroff space and its closed points are
$M=\{(2\lambda_1,\cdots,2\lambda_n):\lambda_1,\ldots,\lambda_n\in\mathbb{Z}\}$,
so $f:\mathcal{K}^n\to \mathcal{K}^n$ has closed graph if and only if there exist
integers $\lambda_1,\ldots,\lambda_n$ such that $f$ is the map with constant value
$(2\lambda_1,\cdots,2\lambda_n)$. So there are infinite countable self--map on $\mathcal{K}^n$
which have closed graph. However $g:\mathcal{K}^n\to \mathcal{K}^n$ with constant
value $(1,\cdots,1)$ is continuous and it does not have closed graph \cite{attar}.
\end{example}
%%%%%%%%%%%%%%%%%%%%%%%%%%%%%%%%%%%%
%%%%%%%%%%%%%%%%%%%%%%%%%%%%%%%%%%%%
%%%%%%%%%%%%%%%%%%%%%%%%%%%%%%%%%%%%
%%%%%%%%%%%%%%%%%%%%%%%%%%%%%%%%%%%%
%\section*{Acknowledgement}
%\noindent ??
%%%%%%%%%%%%%%%%%%%%%%%%%%%%%%%%%%%%

\noindent {\small $\:$ \\
{\bf Fatemah Ayatollah Zadeh Shirazi}, Faculty
of Mathematics, Statistics and Computer Science, College of
Science, University of Tehran, Enghelab Ave., Tehran, Iran
\\
(f.a.z.shirazi@ut.ac.ir)}
\\ $\:$ \\
{\small {\bf Sajjad Moradi Chaleshtori}, Faculty
of Mathematics, Statistics and Computer Science, College of
Science, University of Tehran, Enghelab Ave., Tehran, Iran
\\
(smoradi.ch@ut.ac.ir)}

%
%%%%%%%%%%%%%%%%%%%%%%%%%%%%%%%%%%%%
% 


\begin{thebibliography}{99}

\bibitem{alex}  P. Alexandroff, {\it Diskrete R\"aume}, Mat. Sb., 2 (1937), 501--518.

\bibitem{T2} I. Baggs, {\it Properties of functions with a closed graph}, Topol. Appl., Proc. Conf. St. John’s 1973 (1975), 125--131.

\bibitem{gr} W. J. Baker, {\it Topological groups and the closed--graph theorem}, J. Lond. Math. Soc. 42 (1967), 217--225.

\bibitem{bana} T. Banakh, M. Filipczak, J. W´odka, {\it Returning functions with closed graph are continuous}, Mathematica Slovaca, 70, no. 2 (2020), 297--304.

\bibitem{husain} T. Husain, {\it On a closed graph theorem for topological groups}, Proc. Japan Acad. 44 (1968), 446--449.

\bibitem{tl} M. Nakamura, {\it On closed graph theorem}, Proc. Japan Acad. 46 (1970), 846--849.

\bibitem{attar} M. Pourattar, F. Ayatollah Zadeh Shirazi, M. R. Mardanbrigi, {\it Closed graph property and Khalimsky spaces}, J. Finsler Geom. Appl. 6, No. 1 (2025), 76--82.

\end{thebibliography}
\end{document}